\documentclass[a4paper,11pt]{amsart}
\usepackage{hyperref}

\usepackage[usenames,dvipsnames,svgnames,table]{xcolor}
\usepackage[all]{xy}
\usepackage{enumitem}
\usepackage[normalem]{ulem}
\newcommand{\tsk}[1]{\textcolor{YellowOrange}}

\usepackage{tikz}
\usepackage{tikz-cd}
\usepackage{graphicx}
\usepackage{pgf,tikz}
\usetikzlibrary{arrows}
\usepackage[top=1.2in,bottom=1.2in,left=1.in,right=1.in]{geometry} 

\makeatletter
\def\@endtheorem{\endtrivlist}
\makeatother

\usepackage[active]{srcltx}
\usepackage{amsmath}
\usepackage{amssymb}
\usepackage{amscd}
\usepackage{amsthm}
\usepackage{mathrsfs}  

\usepackage{nicefrac}

\newtheorem{innercustomthm}{Theorem}
\newenvironment{customthm}[1]
  {\renewcommand\theinnercustomthm{#1}\innercustomthm}
  {\endinnercustomthm}

\newtheorem{teo}{Theorem}[section]
\newtheorem*{teo*}{Theorem} 

\newtheorem{defin*}{Definition}
\newtheorem{prop}[teo]{Proposition}
\newtheorem*{prop*}{Proposition} 

\newtheorem*{cor*}{Corollary} 
\newtheorem{lemma}[teo]{Lemma}

\theoremstyle{definition}

\newtheorem{remark}[teo]{Remark}
\newtheorem{ex}{Example}

\newtheorem*{conj*}{Conjecture} 

\newtheoremstyle{dico}
{\baselineskip}   
{\topsep}   
{}  
{0pt}       
{} 
{.}         
{5pt plus 1pt minus 1pt} 
{}          
\theoremstyle{dico}
\newtheorem{say}[teo]{}
\numberwithin{equation}{section}

\newcommand{\ra}{\rightarrow}
\newcommand{\C}{\mathbb{C}}
\newcommand{\R}{\mathbb{R}}
\newcommand{\Zeta}{{\mathbb{Z}}}

\newcommand{\QQ}{{\mathbb{Q}}}

\newcommand{\End}{\operatorname{End}}

\renewcommand{\phi}{\varphi}

\newcommand{\Gl}{\operatorname{GL}}

\newcommand{\PP}{\mathbb{P}^1}   
\newcommand{\V}{\mathbb{V}}

\newcommand{\sieg}{\mathfrak{S}}

\newcommand{\Hg}{\operatorname{Hg}}

\newcommand{\U}{\operatorname{U}}
\newcommand{\SU}{\operatorname{SU}}
\newcommand{\Sp}{\operatorname{Sp}}

\newcommand{\Mg}{\mathcal{M}_g}
\newcommand{\A}{\mathcal{A}}
\newcommand{\Z}{\mathsf{Z}}
\newcommand{\M}{\mathsf{M}}

\newcommand{\ag}{\mathcal{A}_g}
\newcommand{\mg}{\mathcal{M}_g}

\newcommand{\B}{\mathsf{B}}
\newcommand{\mihi}[1]{}

\newcommand{\Mon}{\operatorname{Mon}}
\newcommand{\Hdg}{\operatorname{Hg}}

\begin{document}

\pagestyle{myheadings}

\title{Decomposable abelian $G$-curves and special subvarieties}


\author{Irene Spelta}
\address{Institut f\"ur Mathematik, Humboldt-Universit\"at zu Berlin, Unter den Linden 6, 10099 Berlin,
Germany.}
\email{irene.spelta@hu-berlin.de}
\author{Carolina Tamborini}
\address{Essener Seminar f\"ur Algebraische Geometrie und Arithmetik of Universit\"at Duisburg-Essen, Thea-Leymann-Str. 9 45127 Essen, Germany.}
\email{carolina.tamborini@uni-due.de}

\thanks{
	The authors are members of GNSAGA (INdAM) and are partially supported by INdAM-GNSAGA project CUP E55F22000270001. I. Spelta is partially supported by the Catalan research project SGR 00697, while C. Tamborini is partially supported by the Dutch Research Council NWO and by DFG-Research Training Group 2553
	“Symmetries and classifying spaces: analytic, arithmetic, and derived".}

\begin{abstract}
We consider families of abelian Galois coverings of the line.
When the Jacobian of the general element is totally decomposable, i.e., is isogenous to a product of elliptic curves, we prove that they yield special subvarieties of $\A_g$ if and only if a numerical condition holds, which in the general case is only known to be sufficient. 

\end{abstract}
\maketitle

\section{Introduction}
\begin{say}
	Let $\A_g$ be the moduli space of principally polarized abelian varieties of dimension $g$ over $\C$. The \emph{special} subvarieties of $\A_g$ are defined as Hodge loci for the natural variation of Hodge structure over $\QQ$ on $\A_g$, whose fiber over $A$ is $H^1(A, \QQ)$. The moduli space $\A_g$ is the quotient of the Siegel space $\sieg_g$ by the action of $\Sp(2g, \Zeta)$.  We denote by $\pi: \sieg_g \rightarrow \A_g$ the natural projection map. An algebraic subvariety $Z\subset \A_g$ is \emph{totally geodesic} if $Z=\pi(X)$ for some (connected) totally geodesic submanifold $X\subset \sieg_g$. We say that $X$ is the symmetric space \emph{uniformizing} $Z$. Special subvarieties of $\A_g$ are totally geodesic. Moreover, a totally geodesic subvariety $Z$ of $\A_g$ is special if and only if it contains a complex multiplication (CM) point (\cite{mumford-Shimura, moonen-linearity-1}). Let $\Mg$ denote the moduli space of smooth, complex algebraic curves of genus $g$ and consider the Torelli map $j: \Mg \ra \A_g$. By the Torelli theorem, the map $j$ is injective. A well known conjecture, due to Coleman and Oort, studies the interplay between the image of the Torelli map and the special subvarieties of $\A_g$. Specifically, it predicts that, for $g\gg0$, there are no positive dimensional special subvarieties $Z$ of $\A_g$ such that $Z\subset \overline{j(\Mg)}$ and $Z\cap j(\Mg)\neq \emptyset$.
\end{say}

\begin{say}\label{fam}
For low genus there exist \emph{counterexamples} to the Coleman-Oort conjecture, i.e. there exist examples of (positive-dimensional) special subvarieties $Z$ of $\A_g$ such that $Z\subset \overline{j(\Mg)}$ and $Z\cap j(\Mg)\neq \emptyset$. These have been object of study of many authors, we mention e.g., \cite{dejong-noot, fgp, fpp, fgs, gm1, mohajer-zuo-paa, moonen-special, moonen-oort, ire}. All the examples known so far are in genus $g\leq 7$ and are obtained via families of Galois covers of curves. The idea of the main construction is the following. Let $G$ be a finite group and let $\mathcal{C}\rightarrow \B$ be the family of all Galois covers $C_t\rightarrow C'_t=C_t/G$, where the genera $g=g(C_t)$, $g'=g(C_t')$, the ramification and the monodromy are fixed. Let $\M$ be the image in $\Mg$ of $\B$ and let $\Z:=\overline{j(\M)}$. Let $S(G)\subset \A_g$ be the PEL special subvariety associated with $G$. Clearly we have $\Z\subseteq S(G)$. In \cite{moonen-special} in the case $g'=0$ and $G$ cyclic and in \cite{fgp, fpp} for any $G$ and $g'$, it is proven that if the numerical condition
	\begin{gather}\tag{$\star$}\label{star}
		\dim \Z= \dim H^0(C, 2K_{C})^G=\dim (S^2H^0(C, K_{C}))^G=\dim S(G)
	\end{gather}
	holds true for a general $[C]\in \M$, then $\Z$ is a special subvariety of $\A_g$. Clearly, $\Z$ lies in the Torelli locus and meets the open Torelli locus non-trivially. Thus, under the condition \eqref{star}, $\Z$ provides a counterexample to the conjecture. 
\end{say}

\begin{say}\label{casi trattati} 
	It is a natural problem to understand whether \eqref{star} is also necessary for a family of $G$-covers to yield a special subvariety. We briefly recall what is known about this problem. In \cite{moonen-special} Moonen proved the necessity of condition \eqref{star} when $g'=0$ and the group $G$ is cyclic.  His proof relies on deep results in positive characteristics. Making use of similar techniques, Mohajer and Zuo \cite{mohajer-zuo-paa} extended this to the case where $g'=0$, the group $G$ is abelian and the family is one-dimensional. Finally, still using the reduction modulo $p$ technique, Mohajer proved in \cite{Moha} the necessity of \eqref{star} for $g'=0$, $G$ abelian, plus some extra condition on the monodromy of the family. A Hodge theoretic argument is given in \cite[Prop. 5.7]{cfg}. Here, Colombo, Frediani, and Ghigi proved that \eqref{star} is necessary for $\Z$ to be totally geodesic in the case $g'=0$, $G$ cyclic, plus another condition on the dimension of the eigenspaces for the representation of $G$ on $H^0(C, K_C)$. In \cite{f}, Frediani generalized this result to the case $g'=0$ and $G$ abelian with the analogous condition on the eigenspaces. 
\end{say}

\begin{say}
	In this work, we consider the case where $G$ is abelian and the Jacobian of the general element of the family is totally decomposable, i.e. it is isogenous to a product of elliptic curves. Our main result is the following:

\begin{customthm}{A}\label{ThmA}
	Let $\mathcal{C}\rightarrow \B$ be a totally decomposable family of abelian coverings of the line, then $\Z$ is special if and only if \eqref{star} holds.
\end{customthm}

We stress that, by results of Mohajer and Zuo, the necessity of \eqref{star} was already known in the abelian case for one-dimensional families with $g'=0$. Our result allows to take into account totally decomposable abelian coverings of the line of any dimension. {Also, note that families of totally decomposable abelian $G$-covers do not fall within the aforementioned cases already discussed by Mohajer \cite{Moha} and Frediani \cite{f}. In this way, our result goes in the direction of completing the proof of the necessity of \eqref{star} in the abelian case.}
The result fits into a more general approach to the Coleman-Oort conjecture for totally decomposable special subvarieties. This was first addressed by Moonen and Oort in \cite[Question 6.6]{moonen-oort}. Moreover, Lu and Zuo in \cite{lz} showed  that for $g>11$, there do not exist special curves such that the Jacobian of the generic point decomposes as $E^k$. 
\end{say}

\begin{say}
The work is organized as follows: in Section \ref{Gcurves} we review the basics on families of $G$-curves. Section \ref{smallest special} is devoted to the description of the generic Hodge group $\Hg$ of a family of abelian coverings of the line. This is fundamental because the orbit of the generic Hodge group provides the smallest special subvariety $S_f$ of $\A_g$ containing the family and the problem of understanding whether the condition \eqref{star} is necessary relies exactly on understanding the inclusions $\Z\subseteq S_f\subseteq S(G)$. In this setting, we characterize the relation between the special subvarieties $S_f$ and $S(G)$ when the symmetric space uniformizing $S_f$ is isomorphic to a product of $\sieg_1$'s (which is the case for totally decomposable families). Indeed, we prove the following:
\begin{prop}\label{graph0}
	Let $\mathcal{C}\rightarrow \B$ be a family of abelian $G$-coverings of the line. Suppose that $\Z=S_f$ and that the uniformizing symmetric space $M_f$ of $S_f$ is isomorphic to a product of $\sieg_1$'s. Then the same is true for the unifomizer $M(G)$ of $S(G)$, namely we have:
	 \begin{gather}
	 	M_f\cong\underbrace{\sieg_1\times\dots \times\sieg_1}_{r \text{ times}} \hookrightarrow \underbrace{\sieg_1\times\dots \times\sieg_1}_{n \text{ times}}\cong M(G)
	 \end{gather}
	with $n\geq r$. Moreover, $\Z=S_f\subsetneq S(G)$ if and only if one of the $\sieg_1$'s in $M_f$ maps into $M(G)$ as the graph of an isometry $h:\sieg_1\ra \sieg_1$
			 \begin{gather*}
		M_f \supseteq \sieg_1 \hookrightarrow  \sieg_1\times \sieg_1\subseteq M(G) \\
	\tau \longmapsto (\tau, h(\tau)).
	\end{gather*}
	
		
\end{prop}
(See Proposition \ref{graph}). With similar techniques, along with this Proposition, we also prove the following related result of independent interest.

\begin{teo}\label{tutti diversi0}
   	Let $\mathcal{C}\rightarrow \B$ be a family of abelian $G$-coverings of the line. Suppose that 
	\begin{gather*}
		\SU(m_{\chi_i}, m_{\bar\chi_i})\neq \SU(m_{\chi_j}, m_{\bar\chi_j})
	\end{gather*}
	for all the irreducible complex representations $\chi_i, i\neq j$ of $G$, with multiplicities $m_{\chi_i} m_{\bar {\chi}_i}\neq 0$ on the cohomology of the fibers. Then $S_f=S(G)$. In particular, $\Z$ is special if and only if \eqref{star} holds. 
\end{teo}
(See Theorem \ref{tutti diversi}). In Section \ref{proofmain} we consider families of totally decomposable abelian $G$-curves, and we prove the Theorem \ref{ThmA}. Indeed, by describing the Hodge group for a totally decomposable Jacobian, we show that, when the family is special, the graph-situation in Proposition \ref{graph0} can not occur, implying that, in this case, $S_f=S(G).$
\end{say}
\vspace{4 mm}

{\bfseries \noindent{Acknowledgements}} We heartily thank Alessandro Ghigi for interesting discussions and useful suggestions.

\section{Families of $G$-curves}\label{Gcurves}	
This section is devoted to recalling basic facts and properties of families of Galois coverings. In particular, we describe when they yield a special subvariety of $\A_g$. 	
\begin{say}\label{condizionestar}
	Let $G$ be a finite group and let $\mathcal{C}\rightarrow \B$ be the family of all Galois covers $C_t\rightarrow C'_t=C_t/G$, where the genera $g=g(C_t)$, $g'=g(C_t')$, the ramification, and the monodromy are fixed  (see \cite{ganzdiez, GT2} for details on the construction of such families). 
    Let $\M$ be the image in $\mg$ of $\B$ and let $\Z:=\overline{j(\M)}$. 	If $C_t$ is one of the curves in the family, the action of $G$ on $C_t$ induces an action of $G$ on holomorphic one-forms:
	\begin{gather}\label{azioneforme}
		\rho: G\ra \Gl(H^0(C_t, K_{C_t})), \quad \rho(g)(\omega)=g.\omega:=(g^{-1})^*(\omega).
	\end{gather}
	Notice that the equivalence class of $\rho$ does not depend on the point $t\in \B$. Since $G$ acts holomorphically on $C_t$, $\rho$ maps $G$ injectively into $\Sp(2g, \R)$. 
	In particular, $G$ is a group of isometries of $\sieg_g$ and, since $G$ is finite, $\sieg_g^G$ is non-empty. It follows that $\sieg_g^G$ is a smooth connected totally geodesic submanifold of $\sieg_g$. The image $S(G)$ of $\sieg_g^G$ in $\A_g$ is a special subvariety of $\A_g$ of PEL type (see \cite{fgp, fpp}). By construction, $S(G)$ contains $\Z$. The condition
	\begin{gather}\tag{$\star$}\label{star}
		\dim \Z=\dim S(G),
	\end{gather}
	implies that $\Z=S(G)$ and hence it ensures that $\Z$ is special in $\A_g$. In particular, it provides a counterexample to the Coleman-Oort conjecture.  The orbifold cotangent space of $\M$ in $x=[C_t]$ is $H^0(C_t, 2K_{C_t})^G$, while the orbifold cotangent space of $S(G)$ at $j(x)=[JC_t]$ is $T^*_{j(x)}(\sieg_g)^G=(T^*_{j(x)}\sieg_g)^G=(S^2H^0(C_t, K_{C_t}))^G$. It follows that we can read condition \eqref{star} as:
	\begin{gather}\tag{$\star$}\label{star}
		\dim H^0(C_t, 2K_{C_t})^G=\dim (S^2H^0(C_t, K_{C_t}))^G.
	\end{gather}
\end{say}

\begin{say}\label{storiellaesempi}
	Using condition \eqref{star}, various counterexamples to the conjecture have been obtained. These are in genus $g\leq 7$ and can be divided in two classes: 
	\begin{enumerate}
		\item those obtained as families of Galois covers satisfying \eqref{star};
		\item those obtained via fibrations constructed on the examples in $(1)$ with $g'=1$ (see \cite{fgs}).
	\end{enumerate}
	It is proved in \cite{fgs} that there are no families of Galois covers of curves with $g'\geq 2$ satisfying $(\star)$. Table 2 in \cite{fgp} lists all known examples of type $(1)$ with $g'=0$. Table 2 in \cite{fpp} lists all known examples of type of type $(1)$ with $g'=1$. Actually, in \cite{cgp}, the authors show that the ones listed in \cite{fgp, fpp} are the only positive-dimensional families of Galois coverings satisfying $(\star)$ with $2\leq g\leq 100$. In \cite{ct} the symmetric spaces uniformizing each of these counterexamples have been computed. Of a different nature are the two examples in \cite{ire}, obtained via non-Galois coverings.
	
\end{say}

\begin{say}\label{sfSG}
	Of course, the condition \eqref{star} is a priori just a sufficient condition: when $\dim \Z < \dim S(G)$ one can't conclude that $\Z$ is not special. It may happen that $\Z$ is a smaller special subvariety contained in $S(G)$. In the following we will denote by $S_f$ the smallest special subvariety containing $\Z$. We thus have $\Z\subseteq S_f \subseteq S(G)$ and $\Z$ is special if and only if $\Z=S_f$. To prove the necessity of \eqref{star} means to prove that $\Z$ is special if and only if $\Z= S_f = S(G)$.
\end{say}

\section{Hodge group and smallest special subvariety}\label{smallest special}
In this section we give a summary of the known results about the generic Hodge group of a family of abelian $G$-curves.
\begin{say}
	Let $G$ be a finite abelian group. As in \ref{condizionestar}, consider the family $f:\mathcal{C}\rightarrow \B$ of all Galois covers $C_b\rightarrow C'_b=C_b/G$ with fixed genera $g=g(C_b)$, $g'=g(C_b')$, ramification and monodromy. Let $\V= R^1f_*\mathbb{Q}$ be the local system of the natural $\mathbb{Q}$-VHS associated with the family, and $b\in B$ be a Hodge-generic point with respect to the variation $\V$. The $\QQ$-VHS comes equipped with a $G$-action. Let $X(G)$ denote the set of $\mathbb{C}$-irreducible representations of $G$. We have the following decomposition in sublocal systems
	\begin{gather*}
		\V_{\mathbb{C}}=\oplus_{\chi \in X(G)}  \V_{\chi},
	\end{gather*}
	where $\V_{\chi}$ denotes the $\chi$-isotypical summand over $\mathbb{C}$. At the point $b$, we have that $\V_b=H^1(C_b, \QQ)$. After tensoring with $\C$, we have $H^1(C_b, \C)=H^{1,0}(C_b)\oplus H^{0,1}(C_b)$. In particular, the first term decomposes as \begin{equation}\label{H0kC}
	    H^{1,0}(C_b)=  \oplus_{\chi \in X(G)}  m_{\chi} V_{\chi},
	\end{equation} where $V_{\chi}$ is the $\C$-irreducible representation associated with $\chi$ and $m_{\chi}$ its multiplicity. Note that $\V_{\chi, b}=H^1(C_b, \C)_\chi=m_{\chi}V_{\chi}\oplus\overline{m_{\chi}V_{\chi}}= m_{\chi}V_{\chi}\oplus{m_{\bar\chi}V_{\bar\chi}}$. Similarly, let $X(G, \R)$ denote the set of $\mathbb{R}$-irreducible representations of $G$. For $\chi\in X(G)$, we denote by $(\chi, \overline{\chi})$ the real representation associated with it. The analogous decomposition with real coefficients is:
	\begin{gather*}
		\V_{\mathbb{R}}=\bigoplus_{(\chi, \overline{\chi})\in X(G, \R), } \V_{(\chi,\overline{\chi})} .
	\end{gather*} 
	Where, if $\chi\neq \overline{\chi}$, the relation is
	\begin{gather*}
		\V_{(\chi,\overline{\chi})} \otimes_{\R} \C= \V_{\chi} \oplus \V_{\overline{\chi}}
	\end{gather*}
	while if $\chi=\overline{\chi}$, we have
	\begin{gather*}
		\V_{(\chi,\chi)} \otimes_{\R} \C= \V_{\chi}.
	\end{gather*}
	When $\chi\neq \overline{\chi}$, on the summand $\V_{(\chi,\overline{\chi})}$, the polarization of the VHS induces an hermitian form of signature $(m_\chi, m_{\overline{\chi}})$. In the case $\chi=\overline{\chi}$ the polarization induces a symplectic form on $\V_{(\chi,\chi)}$ (see \cite[Corollary 2.21-2.23]{dm}). 
	
	Let $\Mon^0\subset \Gl(H^1(C_b, \QQ))$ and $\Hdg \subset \Gl(H^1(C_b, \QQ))$ be the connected monodromy group and the Hodge group, respectively. In particular, $\Hdg$ is the generic Hodge group of the VHS. By Andr\'e \cite[Theorem 1]{andre}, $\Mon^0$ is a semi-simple normal subgroup of $\Hdg$. We have $\Mon^0\subset \Hdg\subset \Sp(H^1(C_b, \QQ))^G$, where $\Sp(H^1(C_b, \QQ))^G$ is the centralizer of $G$ in $ \Sp(H^1(C_b, \QQ))$. Extending scalars to $\R$ there are the following isomorphisms:
	\begin{gather}\label{decomposizione SG}
		\Sp(H^1(C_b, \R))^G \simeq \prod_{(\chi, \bar\chi )\in X(G, \R)} \Sp(H^1(C_b, \R)_\chi)^G\simeq  \prod_{(\chi, \bar\chi ), \chi\neq\bar\chi} \U(m_\chi, m_{\bar \chi})\times\prod_{(\chi, \bar\chi ), \chi=\bar\chi}\Sp(2m_\chi, \R).
	\end{gather}
	We denote by $\Mon^0(\chi)$ the projection of $\Mon^0_\R$ onto the $\chi$-factor of the decomposition.  In the case of abelian Galois coverings of $\mathbb{P}^1$, Mohajer and Zuo (extending results in \cite{moonen-special,rohde}) proved the following:
\end{say}
\begin{prop}\cite[Lemma 6.4]{mohajer-zuo-paa}\label{mohajerzuo}
	If $m_\chi m_{\bar \chi}\neq 0$, then $\Mon^0(\chi)=\SU(m_\chi, m_{\bar \chi})$. Moreover, if $\chi=\bar \chi$, then $\SU(m_\chi, m_{ \chi})=\Sp_{2m_\chi}.$
\end{prop}
In the proposition and in what follows, we use the slight imprecise notation $\SU(m_\chi, m_{ \chi})=\Sp_{2m_\chi}$, meaning that $\Sp_{2m_\chi}\otimes \C= \SU(m_\chi, m_{ \chi})$ and we are taking the real points.

\begin{say}
	{Let $S_f$ be the smallest special subvariety of $\ag$ that contains $\Z$ (defined as in \ref{condizionestar}). This is the special subvariety associated with $\Hdg$, i.e., it is the special subvariety whose uniformizer is the symmetric space associated with the real group $\Hdg_\R$. This follows from the general fact that special subvarieties in $\A_g$ are uniformized by the symmetric space associated with the generic Hodge group (see \cite[Section 3]{deba} for a thorough survey). More details on this are also recalled in the following section (see Theorem \ref{deba}).} Let $\Hdg_\R^{ad}= Q_1\times\cdots\times Q_s$ be the decomposition of the adjoint Hodge group $\Hdg_\R^{ad}$ in $\R$-simple groups.  If $Q_i$ is non-compact we denote by $\delta(Q_i)$ the dimension of the corresponding symmetric space. When $Q_i$ is compact we set $\delta(Q_i)=0$. Therefore $\dim S_f=\sum \delta(Q_i)$. Obviously, $\Z$ is special if and only if $\dim \Z=\dim S_f$. 
\end{say}

\begin{say}\label{vertaasoluta}
	As we previously observed, $\Mon^0$ is a normal subgroup of $\Hdg$. By semi-simplicity it follows that $\Mon^{0,ad}_\R= \prod_{i\in I} Q_i$, for $I\subseteq \{1,\dots, s\}$. 
    By the decomposition \eqref{decomposizione SG}, we have $\prod_{i\in I} Q_i=\Mon^{0,ad}_\R\subseteq \prod_\chi \Mon^{0,ad}(\chi) $. 
	In the case of a family of abelian coverings of the line, by Proposition \ref{mohajerzuo}, we have $\Mon^{0,ad}(\chi)=\SU(m_\chi, m_{\bar\chi})$. Therefore, for $i\in I$, since $Q_i$ is normal in  $\prod_{i\in I}Q_i$, and the projection of $\Mon^{0,ad}_\R= \prod_{i\in I}Q_i$ onto $\Mon^{0,ad}(\chi)=\SU(m_\chi, m_{\bar\chi})$ is surjective, we get that $Q_i$ is normal in $\SU(m_\chi, m_{\bar\chi})$. By simplicity,  we deduce that, for $i\in I$, $Q_i=\Mon^{0,ad}(\chi)=\SU(m_\chi, m_{\bar\chi})$ for some $\chi$.
	Thus, we have 
	\begin{gather}\label{prodmon}
		\Mon^{0,ad}_\R\hookrightarrow \prod_{m_\chi m_{\bar \chi}\neq 0} \SU(m_\chi, m_{\bar\chi})
	\end{gather}
	and each natural projection is surjective. Moreover, when $\Z$ is special, $\Mon^0=\Hdg^{der}$ (see for instance \cite[Theorem 3.1.4]{rohde}). Thus, in particular, $\Mon^{0,ad}_\R=\Hdg_\R^{der,ad}=\Hdg_\R^{ad}$. Hence, when $\Z$ is special, we have
	\begin{gather*}
		\Mon^{0,ad}_\R=\Hdg_\R^{ad}= \prod_{i=1}^sQ_i \hookrightarrow \prod_{m_\chi m_{\bar \chi}\neq 0} \SU(m_\chi, m_{\bar\chi}).	
	\end{gather*} 
	Therefore, $\Hdg_\R^{ad}\neq \prod_{m_\chi m_{\bar \chi}\neq 0} \SU(m_\chi, m_{\bar\chi})$ if and only if there exist $\chi_1\neq \chi_2$ such that $\Hdg_\R^{ad}$ contains a subgroup isomorphically mapped to $\SU(m_{\chi_1}, m_{\bar\chi_1})=\SU(m_{\chi_2}, m_{\bar\chi_2})$. 
\end{say}

\begin{say}\label{differenza}
	In order to address the issue discussed in \ref{sfSG}, let us recall that $\Sp(H^1(C_b, \R))^G$ is the group associated with the uniformizing symmetric space $\sieg_g^G$ of the PEL special subvariety $S(G)$. Using \eqref{decomposizione SG}, we have $\dim S(G)=\sum \delta(SU(m_\chi, m_{\bar{\chi}}))$. Therefore $\Z=S_f\subsetneq S(G)$ if and only if $\dim S_f < \dim S(G)$, hence if and only if $\Hdg_\R^{ad}\neq \prod_{m_\chi m_{\bar \chi}\neq 0} \SU(m_\chi, m_{\bar\chi})$.
\end{say}

\ \\
Summing up, we have proved the following:

\begin{prop}\label{riassuntino} 	Let $\mathcal{C}\rightarrow \B$ be a family of abelian coverings of the line. Suppose that $\Z=S_f$. The generic Hodge group of the family has the following properties:
	\begin{enumerate}
		\item $\Hg_{\R}^{ad}\cong Q_1\times\cdots\times Q_s$, with $Q_i=\Mon^{0,ad}(\chi)=SU(m_\chi, m_{\bar\chi})$ for some $\chi$;
		\item $\Hdg_\R^{ad}\hookrightarrow \prod_{m_\chi m_{\bar \chi}\neq 0} \SU(m_\chi, m_{\bar\chi})$ and each natural projection is surjective;
		\item  $\Z=S_f\subsetneq S(G)$ if and only if $\Hdg_\R^{ad}\neq \prod_{m_\chi m_{\bar \chi}\neq 0} \SU(m_\chi, m_{\bar\chi})$, namely if and only if there exist $\chi_1\neq \chi_2$ such that $\Hdg_\R^{ad}$ contains a subgroup isomorphically mapped to $\SU(m_{\chi_1}, m_{\bar\chi_1})=\SU(m_{\chi_2}, m_{\bar\chi_2})$.
	\end{enumerate}
\end{prop}
The problem of understanding whether the condition \eqref{star} is necessary relies on the problem of understanding what is the smallest special subvariety $S_f$ containing the given a family of abelian $G$-covers. Under an assumption on the eigenspaces, we can state the following.

\begin{teo}\label{tutti diversi}
   	Let $\mathcal{C}\rightarrow \B$ be a family of abelian $G$-coverings of the line. Suppose that 
	\begin{gather*}
		\SU(m_{\chi_i}, m_{\bar\chi_i})\neq \SU(m_{\chi_j}, m_{\bar\chi_j})
	\end{gather*}
	for all $\chi_i, \chi_j\in X(G), i\neq j$ with $m_{\chi_i} m_{\bar {\chi}_i}\neq 0$. Then $S_f=S(G)$. In particular, $\Z$ is special if and only if \eqref{star} holds. 
\end{teo}

\begin{proof}
Under the assumption, \eqref{prodmon} implies that 
\begin{gather*}
    \Mon^{0,ad}_\R= \prod_{m_\chi m_{\bar \chi}\neq 0} \SU(m_\chi, m_{\bar\chi})\subseteq \Hdg_{\R}^{ad}.
\end{gather*}
Hence we get $\dim S(G)=\sum \delta(SU(m_\chi, m_{\bar{\chi}}))\leq \dim S_f\leq \dim S(G)$, which  implies $S_f=S(G)$. In particular, $\Z$ is special if and only if $\Z=S_f=S(G)$.
\end{proof}
To provide a better understanding of the results above, now we discuss some examples. As in section \ref{Gcurves}: $\mathcal{C}\rightarrow \B$ is the family of abelian $G$-coverings of $\PP$ branched on $s$  points, the vector $\Theta=(\theta_1, \dots, \theta_s)$ gives the monodromy, and $\bar m= (\operatorname{ord} \theta_1, \dots, \operatorname{ord}\theta_s) $ the local orders. Clearly, we have $\dim B= s-3$. Let us also recall that, since the group is abelian, the irreducible complex representations $\chi \in X(G, \C)$ are in 1-1 correspondence with the elements of $G$. Thus, we denote them as $\chi_g, g \in G$.
\begin{ex}
    This first example is an application of Theorem \ref{tutti diversi}. 
The data are the following: \begin{gather*}
    G=\Zeta/6; \quad  \Theta=([3],[3],[3],[4],[5]) \:\:(s=5, g=4)\\ \bar m= (2,2,2,3,6).
\end{gather*}
The dimensions $m_{\chi}, m_{\bar\chi}$ of the eigenspaces of the  $G$-action on the cohomology of the fibers can be computed in terms of the monodromy datum (see e.g. \cite[Prop. 2.8]{mohajer-zuo-paa}). We get the following: 
\begin{gather*}
    \Mon^{0, ad}(\chi_{[1]})=\SU(1,2),\, \Mon^{0, ad}(\chi_{[2]})=\{1\},\, \Mon^{0, ad}(\chi_{[3]})=SU(1,1).
\end{gather*}
In this situation, by Theorem \ref{tutti diversi}, we conclude that $\Z\subsetneq S_f=S(G).$
\end{ex}
\begin{ex}
In this second example we present a situation where Proposition \ref{riassuntino} allows us to understand that the family is not special. Nevertheless, we can not decide whether $S_f=S(G)$. The data are the following: \begin{gather*}
    G=\Zeta/2\times \Zeta/2\times \Zeta/2=\langle g_1\rangle\times \langle g_2\rangle\times \langle g_3\rangle; \quad  \Theta=(g_2, g_2, g_2, g_1g_2, g_1g_2g_3, g_2g_3) \:\:(s=6, g=5)\\ \bar m= (2,2,2,2,2,2).
\end{gather*}
The eigenspaces are the following: 
\begin{gather*}
    \Mon^{0, ad}(\chi_{g_1})=\{1\},\, \Mon^{0, ad}(\chi_{g_2})=\Sp_4,\, \Mon^{0, ad}(\chi_{g_3})=\{1\}, \, \Mon^{0, ad}(\chi_{g_1g_2})=SU(1,1), \\
    \Mon^{0, ad}(\chi_{g_1g_3})=\{1\},\, \Mon^{0, ad}(\chi_{g_2g_3})=SU(1,1),\, \Mon^{0, ad}(\chi_{g_1g_2g_3})=SU(1,1).
\end{gather*}
In this situation, reasoning as in Theorem \ref{tutti diversi}, we can just say that $\Sp_2\times SU(1,1)\subseteq \Hg^{ad}_\R$. This is enough to conclude that $\Z$ is not special, since $\dim\Z=3<4\leq \dim S_f$. On the other hand, unfortunately, we cannot understand whether $S_f\subseteq S(G)$ is just an inclusion or an equality.
\end{ex}
Finally, we discuss two last examples where, even if Proposition \ref{riassuntino} is not enough to conclude, we can use some evidences from the families to decide if they are, or they are not, special.
\begin{ex}
The data are the following:
\begin{gather*}
    G=\Zeta/2\times \Zeta/2=\langle g_1\rangle\times \langle g_2\rangle; \quad  \Theta=(g_2, g_1, g_1, g_1, g_1g_2, g_1, g_1) \:\:(s=7, g=4)\\ \bar m= (2,2,2,2,2,2,2).
\end{gather*}
We have that \begin{equation*}
    \Mon^{0, ad}(\chi_{g_1})=\Sp_4, \;\Mon^{0, ad}(\chi_{g_2})={1},\; \Mon^{0, ad}(\chi_{g_1g_2})=\Sp_4.
\end{equation*}
In principle, Proposition \ref{riassuntino} does not allow us to decide whether $S_f$ is equal or different from $S(G)$, since the two non-trivial monodromy groups are equal. Nevertheless, the symmetric space associated with $\Sp_4$ is 3-dimensional while the family is 4-dimensional, and so $\dim S_f\geq 4$. It follows that $S_f=S(G)(\neq Z)$.
\end{ex}

\begin{ex}
The data are the following:
\begin{gather*}
    G=\Zeta/2\times \Zeta/2=\langle g_1\rangle\times \langle g_2\rangle; \quad  \Theta=(g_2, g_1, g_2, g_1, g_1g_2, g_1, g_2) \:\:(s=7, g=4)\\ \bar m= (2,2,2,2,2,2,2).
\end{gather*}
We have that \begin{equation*}
    \Mon^{0, ad}(\chi_{g_1})=SU(1,1), \;\Mon^{0, ad}(\chi_{g_2})=SU(1,1),\; \Mon^{0, ad}(\chi_{g_1g_2})=\Sp_4.
\end{equation*}
Let us denote by $\lambda_1, \ldots, \lambda_4$ the four parameters varying in the family. By looking at the intermediate quotient curves, we can explain where do these eigenspaces come from. Indeed, the quotient curve $E:=C/\langle g_1\rangle$ is an elliptic curve moving in the Legendre family $y^2=x(x-1)(x-\lambda_1)$. This yields the first $SU(1,1)$. The quotient curve $C':=C/\langle g_1g_2\rangle$ is a genus 2 curve moving on a 3-dimensional locus $C'= C'_{\lambda_2, \lambda_3, \lambda_4}$. This yields $\Sp_4$. Finally, the quotient  $F:=C/\langle g_2\rangle$ is again an elliptic curve moving on a 1-dimensional locus. This yields the second $SU(1,1)$. Thus, by \cite[Theorem 0.1]{moonen-zar}, we get that $\Hdg^{ad}_\R= \Hdg^{ad}_\R(E)\times \Hdg^{ad}_\R(F)\times \Hdg^{ad}_\R(JC') $. In other words, we obtain $S_f=S(G) (\neq Z)$.

\end{ex}
\begin{remark}
    All the examples above suggest that it is reasonable to believe that the equality $S_f=S(G)$ should always hold in case of families of abelian covers of $\PP$. 
\end{remark}

Theorem \ref{tutti diversi} provides a complete description of $S_f$ in the case where the $SU(m_{\chi}, m_{\bar{\chi}})$ are all different. Now we focus on the other extremal case,  namely we consider the situation where the generic Hodge group is a product of $SU(1,1)$. In preparation for the next proposition, we recall the following fact from \cite{ire tesi}:

\begin{lemma}\label{graficosiegel}[\cite{ire tesi}, Prop. 3.2.4]
	Let $X$ be a one-dimensional irreducible totally geodesic submanifold of $\sieg_1\times \sieg_1$ such that $pr_1(X)=\sieg_1$, where $pr_1$ is the projection onto the first factor. Then $X$ is the graph of an isometry $h: \sieg_1 \ra \sieg_1$.
\end{lemma}

\begin{prop}\label{graph}
	Let $\mathcal{C}\rightarrow \B$ be a family of abelian $G$-coverings of the line. Suppose that $\Z=S_f$ and that $\Hg_{\R}^{ad}= Q_1\times\cdots\times Q_s$ with $Q_i\cong SU(1,1)$ for all $i.$ 
	Then both $\Z=S_f$ and $S(G)$ are uniformized by products of $\sieg_1$'s, i.e., if $M_f$ denotes the uniformizer of $S_f$ and $M(G)$ the uniformizer of $S(G)$, we have
	 \begin{gather}\label{unifo fam e pel}
	 	M_f\cong\underbrace{\sieg_1\times\dots \times\sieg_1}_{r \text{ times}} \hookrightarrow \underbrace{\sieg_1\times\dots \times\sieg_1}_{n \text{ times}}\cong M(G)
	 \end{gather}
	with $n\geq r$. Moreover, $\Z=S_f\subsetneq S(G)$ if and only if one of the $\sieg_1$'s in $M_f$ maps into $M(G)$ as the graph of an isometry $h:\sieg_1\ra \sieg_1$
	
		 \begin{gather}\label{funzione}
		M_f \supseteq \sieg_1 \hookrightarrow  \sieg_1\times \sieg_1\subseteq M(G) \\ \notag
	\tau \longmapsto (\tau, h(\tau)).
	\end{gather}
		
\end{prop}
\begin{proof}
    By assumption, we have that all the $Q_i$ appearing in $\Hg_{\R}^{ad}$ are $SU(1,1)$. By Proposition \ref{riassuntino},  $\Hdg_\R^{ad}\subseteq \prod_{m_\chi m_{\bar \chi}\neq 0} \SU(m_\chi, m_{\bar\chi})$ and each natural projection is surjective. Hence for all eigenspaces $\chi$ with $m_\chi m_{\bar \chi}\neq 0$ we have $\SU(m_\chi, m_{\bar\chi})=SU(1,1)$. Thus, we obtain \eqref{unifo fam e pel}. Now suppose that $\Z=S_f\subsetneq S(G)$. Then by point (3) in proposition \ref{riassuntino}, we get that one of the $Q_j=SU(1,1)\subset \Hdg_{\R}^{ad}$ is mapped into $SU(1,1)\times SU(1,1)$, and the projection onto each factor is an isomorphism.  By looking at the associated symmetric spaces, and applying Lemma \ref{graficosiegel}, we prove the second part of the statement.
\end{proof}

\section{Totally decomposable families}\label{proofmain}
In this section we consider families of totally decomposable abelian $G$-curves and we characterize their generic Hodge group. Using this characterization, together with results from Section \ref{smallest special}, we prove Theorem \ref{ThmA}.  We conclude with some comments on the relationship between Theorem \ref{ThmA} and the existing literature.
\begin{say}
	Let $\Lambda$ be a rank $2g$ lattice, $Q: \Lambda \times \Lambda \ra \Zeta$ a symplectic form of type $(1,\ldots,1)$, and set $V:=\Lambda_{\R}$. We denote by $\sieg(V, Q)$ the Siegel space, which can be defined as: \begin{equation}\label{siegel}
		\sieg(V, Q)=\{J\in \Gl(\Lambda_{\R}):\ J^2=-I,\ J^*Q=Q,\ Q(x, Jx)>0,\ \forall x\neq 0\}.
	\end{equation}
	We consider on $\sieg(V, Q)$ the natural polarized integral variation of Hodge structure, whose fiber over $J\in \sieg(V, Q)$ is $V_{\C}$ equipped with the tautological Hodge structure $V_{\C}=V_{-1,0}(J)\oplus V_{0,-1}(J)$. Since this Hodge structure only depends on $J$, we denote it simply by $J$, and we let  $\rho_J: \mathbb{S}(\R) \ra GL(V_{\R})$ be its associated representation, that is $\rho_J(z)v=zv$ for $v\in V_{-1,0}(J)$ and $\rho_J(z)v=\overline{z}v$ for $v\in V_{0,-1}(J)$.
	We denote by $\Hg(J)$ the Hodge group associated with $J$. Note that $\Hg(J)\subset \Sp(\Lambda_{\QQ}, Q)$, since $Q$ is a polarization for $J$. The tautological variation of Hodge structure on $\sieg(V, Q)$ descends to an (orbifold) variation over $\A_g$. We recall the following results (see e.g. \cite[Lemma 4.5]{deba}):
\end{say}	
\begin{teo}\label{deba}
	The orbit $\Hg(J)_{\R}.J\subset \sieg(V,Q)$ is a complex totally geodesic submanifold of $\sieg(V, Q)$. With the induced metric, it is a Hermitian symmetric space of non-compact type. Moreover, assume that $Z\subset \A_g$ is a special subvariety and that $x\in Z$ is general. Then for $J \in\sieg(V, Q)$ with $\pi(J)=x$, we have $Z=\pi(\Hg(J)_{\R}.J)$. 
\end{teo}

\begin{say}
	In the case of elliptic curves, we have:
\end{say}
\begin{prop}\label{prop1}
	Let $E$ be an elliptic curve and $J\in \sieg_1$ such that $\pi(J)=[E]\in \A_1$. Then:
	\begin{enumerate}
		\item if $E$ has CM, then $\Hg(J)_{\R}.J=J$, i.e. it is a point in $\sieg_1$;
		\item otherwise $\Hg(J)_{\R}.J=\sieg_1$.
	\end{enumerate}
\end{prop}
\begin{proof}
	Remember that an elliptic curve $E$ has complex multiplication when $\End(E)\otimes_{\Zeta}\QQ$ is a complex multiplication field, i.e., it is a totally imaginary quadratic extension of a totally real number field. This happens if and only if $\End_{\QQ}(E)\supsetneqq \QQ$.  Equivalently, if $J\in \sieg_1$ is as in the statement, $E$ has complex multiplication if and only if $\Hg(J)$ is a torus algebraic group (\cite[Section 2]{mumford-Shimura}). This happens if and only if $\Hg(J)\subsetneq \Sp(\Lambda_{\QQ}, Q)$. Hence, if $E$ has CM, then $\Hg(J)_{\R}$ is the algebraic torus $(\mathbb{G}_m)_{\R}$, while for $E$ general $\Hg(J)_{\R}=\Sp(2, \R)$. It follows that the orbits are as claimed.
\end{proof}
\begin{remark}\label{diagonally}
	Notice that for $n\geq 1$ one can identify $\Hg(E^n)$ with $\Hg(E)$, acting diagonally on $\Lambda_{\QQ}^n$. 
\end{remark}
In the case of a totally decomposable abelian variety we have the following result by Imai \cite{imai}.
\begin{prop}\label{prop2}
	Let $E_1,\ldots,E_n$ be elliptic curves, no two of which are isogenous. Then we have that  $\Hg(E_1^{k_1}\times\dots \times E_n^{k_n})=\Hg(E_1)\times \dots\times \Hg(E_n)$.
\end{prop}

\begin{say}\label{totdec}
	Let $G$ be a finite group. Consider the family $f:\mathcal{C}\rightarrow \B$ of all Galois covers $C_t\rightarrow \mathbb{P}^1=C_t/G$ with fixed ramification and monodromy. As in \ref{condizionestar}, we denote by $\Z$ the associated subvariety of $\A_g$.
	In the following, we suppose that:
	\begin{enumerate}
		\item $\Z$ is a special subvariety of $\A_g$;
		\item $\Z$ is totally decomposable, i.e., for $x=[JC]\in \Z$ generic, we have $JC\sim E_1^{k_1}\times\dots \times E_n^{k_n}$ with $E_i \nsim E_j$ for $i\neq j$.
	\end{enumerate}
	Since $\Z$ is special, it will contain a dense set of CM points (see \cite{mumford-Shimura}). Let $y\in \Z$ be one of them. By definition, the curves $E_1,\ldots,E_n$ at the point $y$ have CM. It follows that, if ${E}_i$ is not varying in the family, then ${E}_i$ must have CM. Thus, up to reordering, the generic point $JC$ in the family looks like 
    \begin{gather}
    \label{splitting}
       E_1^{k_1}\times\dots \times E_r^{k_r}\times \bar{E}_{r+1}^{k_{r+1}}\times\dots \times \bar{E}_n^{k_n},
    \end{gather} where $\bar{E}_i$, $i=r+1,\dots, n$ are fixed CM elliptic curves. Let $\tau_i\in \sieg_1$ be such that $\pi(\tau_i)=[\bar{E}_i]\in \A_1$. We have the following:
\end{say}
\begin{prop}\label{prop}
	Let $x=[JC]\in \Z$ be a generic point. Then 
	\begin{gather*}
		\Hg(JC)_{\R}\cong\underbrace{\Sp(2,\R)\times\dots \times\Sp(2,\R)}_{r \text{ times}} \times T_{r+1}\times \dots \times T_n	
	\end{gather*}
	where $T_i=(\mathbb{G}_m)_{\R}$. In particular, the symmetric space $\Hg(JC)_{\R}.JC$ associated with $\Z$ is isomorphic to
	\begin{gather}\label{uniformizz}
		\underbrace{\sieg_1\times\dots \times\sieg_1}_{r \text{ times}} \times \tau_{r+1}\times \dots \times \tau_n.
	\end{gather}
	The isomorphism is given by the map
	\begin{gather*}
		\Phi: \underbrace{\sieg_1\times\dots \times\sieg_1}_{r \text{ times}} \times \tau_{r+1}\times \dots \times \tau_n \longrightarrow 	\Hg(JC)_{\R}.JC\subseteq \underbrace{\sieg_1\times\dots \times\sieg_1}_{g \text{ times}} \subset \sieg_g
		\\
		\Phi(J_1,\ldots,J_r, \tau_{r+1},\ldots, \tau_n)=( \underbrace{J_1,\ldots,J_1}_{k_1\  \text{times}},\ldots,\underbrace{J_r,\ldots,J_r}_{k_r\  \text{times}},\underbrace{\tau_{r+1},\ldots,\tau_{r+1}}_{k_{r+1}\  \text{times}} ,\ldots, \underbrace{\tau_n,\ldots,\tau_n}_{k_n\  \text{times}} ).
	\end{gather*}
\end{prop}
\begin{proof}
	By Proposition \ref{prop2} we have $
	\Hg(JC)=\Hg(E_1)\times \dots\times \Hg(E_r)\times\Hg(\bar E_{r+1})\times \dots\times \Hg(\bar E_n)$. Hence, the orbit is isomorphic to $\Hg(JC)_{\R}.JC\cong\Hg(E_1)_{\R}.E_1\times \dots\times \Hg(E_r)_{\R}.E_r\times\Hg(\bar E_{r+1})_{\R}.\bar E_{r+1}\times \dots\times \Hg(\bar E_n)_{\R}.\bar E_n$. By Proposition \ref{prop1} we get \eqref{uniformizz}. The isomorphism is provided by combining this with Remark \ref{diagonally}.
\end{proof}
\textbf{Notation:} From now on we will refer to the orbit space $\Sp(2, \R).E_1\times \dots\times \Sp(2, \R).E_r\times\Sp(2, \R).\bar E_{r+1}\times \dots\times \Sp(2, \R).\bar E_n$ as $\underbrace{\sieg_1\times\dots \times\sieg_1}_{g \text{ times}} $.  
\begin{say}
	When $G$ is abelian, Proposition \ref{riassuntino}, (1) and Proposition \ref{prop} provide two characterizations of the generic Hodge group of the family $f:\mathcal{C}\rightarrow \B$. Now we compare them. 
	
	
\end{say}
\begin{prop}\label{compare}
	When $G$ is abelian, the two decompositions of the generic Hodge group given in Proposition \ref{riassuntino}, (1) and in Proposition \ref{prop} coincide. In particular, $r=s$ and  $\Hg_{\R}^{ad}= Q_1\times\cdots\times Q_r$ with
	\begin{gather}\label{hdgE0}
		Q_i=\Hdg^{ad}(E_i)_{\R}\cong\Mon^{0,ad}(\chi)=SU(1,1),
	\end{gather}
	where $\chi$ is one of the irreducible $\C$-representations of $G$. 
\end{prop}
\begin{proof}
By hypothesis, the family  $f:\mathcal{C}\rightarrow \B$ is a totally decomposable family of $G$-covers of the line such that $G$ is abelian and $\Z$ is special. Hence, it satisfies the assumptions of both Proposition \ref{riassuntino} and Proposition \ref{prop}. As a consequence, the statement follows.
For the reader convenience, we revisit the argument.
	As recalled in section \ref{vertaasoluta}, when $\Z$ is special, by Andr\'e's theorem we have that $\Mon^0=\Hdg^{der}$, where $\Hdg$ is the generic Hodge group of the family. By Proposition \ref{prop2}, we get $\Mon^0=\Hdg^{der}=\Hg(E_1)^{der}\times \dots\times\Hg(E_r)^{der}\times\Hg(\bar E_{r+1})^{der}\times \dots\times \Hg(\bar E_n)^{der}$. Hence $\Mon^0=\Hdg^{der}=\Hg(E_1)^{der}\times \dots\times\Hg(E_r)^{der}$. It follows that $\Mon^{0,ad}_{\R}=\Hdg^{der,ad}_{\R}=\Hg(E_1)^{ad}_{\R}\times \dots\times\Hg(E_r)^{ad}_{\R}$. This is the decomposition in Proposition \ref{prop}. Comparing with the decomposition in Proposition \ref{riassuntino}, (1), we conclude that $Q_i=\Sp(2, \R)=\SU(1,1)$. More precisely, the Hodge group $\Hg(E_i)^{ad}_{\R}$ of each moving elliptic curve $E_i$ is one of the $\Mon^{0,ad}(\chi)$'s appearing in Proposition \ref{riassuntino}, (1).
\end{proof}

\begin{say}
    We now consider the PEL Shimura variety $S(G)$ containing $\Z$. By Proposition \ref{prop}, we know that the symmetric space uniformizing $\Z$ is contained in $\sieg_1^g$. In the following proposition, we show that the same holds for the uniformizer of $S(G).$
\end{say}
\begin{prop}\label{lemma inclusioni}
    Let $JC\in\Z$ be the generic element of the family. Let  $M(G)$ be the  uniformizer of the PEL shimura variety $S(G)$ passing throught $JC$. Then, the following inclusions hold:
	\begin{gather}\label{eq:inclusion}
		\Hg(JC)_{\R}.JC\subseteq M(G)\subseteq\underbrace{\sieg_1\times\dots \times\sieg_1}_{g \text{ times}}.
	\end{gather}
\begin{proof}
By Proposition \ref{prop}, we have that $\Hg(JC)_{\R}.JC\subseteq \underbrace{\sieg_1\times\dots \times\sieg_1}_{g \text{ times}}$. In order to prove the statement, we need to show that the analogous inclusion holds for $M(G)$ as well. Observe indeed that the first inclusion in \eqref{eq:inclusion} follows directly from the definitions.

By Proposition \ref{compare}, we have that $Q_i=SU(1,1)$ of all $i$. Hence, by \eqref{unifo fam e pel}, we have that the uniformizer $M(G)$ of the PEL shimura variety $S(G)$ passing thought $JC$ satisfies
	 \begin{gather}\label{unifo}
	 	\Hg(JC)_{\R}.JC=\underbrace{\sieg_1\times\dots \times\sieg_1}_{r \text{ times}} \subseteq M(G)=\operatorname{H}.JC \cong\underbrace{\sieg_1\times\dots \times\sieg_1}_{n \text{ times}},
	 \end{gather}
where $\operatorname{H}$ is the generic Hodge group of the PEL subvariety $S(G)$. If $n=r$ then $\Hg(JC)_{\R}.JC=\operatorname{H}.JC$, thus the inclusion is obvious. Otherwise, $r<n$. For simplicity we can assume $n=r+1$. By \eqref{funzione} in Proposition \ref{graph}, the injection works as follows: 
\begin{equation}\label{grafico grande}
    \underbrace{\sieg_1\times\dots \times\sieg_1}_{r \text{ times}}\hookrightarrow \underbrace{\sieg_1\times\dots \times\sieg_1}_{r+1 \text{ times}}: (\tau_1, \ldots, \tau_r)\mapsto (\tau_1, \ldots, \tau_r, h(\tau_r)),
\end{equation}
for an isometry $h: \sieg_1\ra \sieg_1$.  In order to describe $M(G)$, let us look at the orbit $\operatorname{H}.JC$ which, by \eqref{grafico grande}, is provided by $\operatorname{H}.( \tau_1, \ldots, \tau_r, h(\tau_r))= \Sp(2, \R).\tau_1 \times \ldots\times\Sp(2, \R).\tau_r\times \Sp(2, \R). h(\tau_r) $.   
The splitting of $M(G)$ as the product of $r+1$ copies of $\sieg_1$ yields a corresponding decomposition (up to isogeny) of the generic element of the PEL as $A_{\tau_1}\times \ldots \times A_{\tau_{r+1}}$. In particular, we obtain $JC\sim B_{\tau_1}\times \ldots \times B_{h(\tau_{r})} $. Since, a priori, $JC$ is not generic in $S(G)$, we have that 
\begin{gather}\label{hodgeBi}
    \Hg(B_i)\subseteq \Hg(A_i).
\end{gather}
On the other hand, since $JC$ is totally decomposable, so are the $B_{\tau_i}'$s, namely they decompose as the product $ \prod E_{i_j}$ of certain of the $E_{i}'$s appearing in \eqref{splitting}. Hence, by Proposition \ref{prop2}, the inclusion \eqref{hodgeBi} gives
\begin{gather}
   \prod \Hg(E_{i_j})= \Hg(B_i)\subseteq \Hg(A_i)=\Sp(2, \R).
\end{gather}
For every $i$ and $j$, $\Hg(E_{i_j})$ is either $\Sp(2,\R)$ or 
$(\mathbb{G}_m)_{\R}$, which forces $B_i$ to be precisely one of the elliptic curves $E_{i_j}$. Therefore we have that $\sieg_1=\Sp(2, \R).B_{\tau_i}=\Hg(A_i). E_{i_j}$. Since this happens for every $i$, this shows that $M(G)$ sits in $\sieg_1^g$. 
\end{proof}
\end{prop}
Now we are ready to show the necessity of the condition \eqref{star} for a family $\mathcal{C}\ra \B$ of totally decomposable abelian covers of the line to be special. \\

\textbf{Proof of Theorem A.} 
Let $[JC]\in\Z$ be a generic point.  By Proposition \ref{lemma inclusioni}, we have that the uniformizer $M(G)$ of the PEL shimura variety $S(G)$ passing throught $JC$ satisfies
	\begin{gather*}
		\Hg(JC)_{\R}.JC\subseteq M(G)\subseteq\underbrace{\sieg_1\times\dots \times\sieg_1}_{g \text{ times}}.
	\end{gather*}
    By Proposition \ref{totdec}, we have that $\Hg(JC)_{\R}.JC$ is the sublocus of $\sieg_1^g$, contained in $M(G)$, parametrized by
\begin{gather*} 
		\{( \underbrace{J_1,\ldots,J_1}_{k_1\  \text{times}},\ldots,\underbrace{J_r,\ldots,J_r}_{k_r\  \text{times}},\underbrace{\tau_{r+1},\ldots,\tau_{r+1}}_{k_{r+1}\  \text{times}} ,\ldots, \underbrace{\tau_n,\ldots,\tau_n}_{k_n\  \text{times}} ) \},
	\end{gather*}
where $\tau_i\in\sieg_1$ are the fixed CM elliptic curves.
By Proposition \ref{graph}, we have that $\Z=S_f\subsetneq S(G)$ if and only if there exists an isometry $h:\sieg_1\ra \sieg_1$  such that  one of the $\sieg_1$'s in $\Hg(JC)_{\R}.JC$ maps into $M(G)$ as the graph of $h$. In other words, in the parametrization of $\Hg(JC)_{\R}.JC$ in $M(G)\subset \sieg_1^g$ appears
\begin{gather*} 
		\{(\ldots, \underbrace{J_i,\ldots,J_i}_{k_i\  \text{times}}, \underbrace{h(J_i),\ldots,h(J_i)}_{k_i\  \text{times}},\ldots) \}.
	\end{gather*}
But this is a contradiction, as these describe two different parametrizations of $\Hg(JC)_{\R}.JC$ inside $\sieg_1^g$.
\qed
\

\begin{say}
	We make some observations regarding the non-abelian case. If $G$ is not abelian, then Proposition \ref{mohajerzuo} does not hold and hence the monodromy groups are not well characterized. Also the PEL subvariety $S(G)$ is more difficult to describe. In the totally decomposable case, we can at least conclude the following.
\end{say}
\begin{prop}
	Let $\Z$ be a totally decomposable subvariety. If $m_{\chi}m_{\bar{\chi}}\neq0$, then $\Mon^0(\chi)$ is isomorphic to $SU(1,1)$ or a product of $SU(1,1)$.
\end{prop}
\begin{proof}
	By Propositions \ref{prop1} and \ref{prop2}, $\Hg^{ad}_{\R}$ is a product of $Q_i=SU(1,1)$. As in \ref{vertaasoluta}, we have that $\prod Q_i\subseteq \prod_\chi \Mon^{0,ad}(\chi) $. By definition, $ \Mon^{0,ad}(\chi) $ is at least semisimple, since $\Hg^{ad}_{\R}$ surjects onto it. 
	Thus we can write it as product of simple factors $\Mon^{0,ad}(\chi)=\prod F_j$. As $Q_i$ is normal in $\Hg^{ad}_{\R}$, its projection is normal as well. Hence we get $Q_i=SU(1,1)=F_j$.
\end{proof}

\subsection{Comparison with previous results}\label{previousres}
As explained in the Introduction \S \ref{casi trattati}, the condition \eqref{star} has been considered by various authors and, in the case of families of abelian coverings of the line, there are some results that prove its necessity under extra conditions on the eigenspaces of the family. In the following we comment on our result in relation with them.

\begin{say}
	Let us start by recalling what is already known about the necessity of \eqref{star} in the case of abelian coverings of the line. 
\end{say}

\begin{teo}[Mohajer and Zuo, Theorem 6.2 \cite{mohajer-zuo-paa}]
	Let $f:\mathcal{C}\rightarrow \B$ be a one-dimensional family of abelian coverings of the line. Then $\Z\subset \A_g$ is special if and only if \eqref{star} holds.	
\end{teo}

\begin{teo}[Mohajer, Theorem 5.2 \cite{Moha}]\label{Moha}
	Let $f:\mathcal{C}\rightarrow \B$ be a $r-$dimensional family of abelian coverings of the line. If there exists $\chi\in X(G)$ such that $\{m_{\chi}, m_{\bar{\chi}}\}\neq \{0, r+1\}$ and $m_{\chi}+m_{\bar{\chi}}\geq r+1$, then $\Z\subset \A_g$ is special if and only if \eqref{star} holds.
\end{teo}

\begin{teo}[Frediani, Theorem 3.1 \cite{f}]\label{Paola}
	Let $f:\mathcal{C}\rightarrow \B$ be a family of abelian coverings of the line. If
	\begin{enumerate}
		\item either there exists $\chi\in X(G)$ with $\chi\neq \bar{\chi}$ such that $m_{\chi}\geq 2$ and $m_{\bar{\chi}}\geq 2$; 
		\item or there exists $\chi\in X(G)$ with $\chi= \bar{\chi}$ such that $m_{\chi}\geq 3$,
	\end{enumerate}
	then \eqref{star} does not hold and $\Z$ is not totally geodesic (hence it is not special).	
\end{teo}

To compare our result with the above ones, we consider our family $f:\mathcal{C}\rightarrow \B$ of totally decomposable abelian $G$-curves of the line.  
When $\Z$ is special, it follows by Proposition \ref{compare} that, for $\chi\in X(G)$, we either have $m_{\chi}m_{\bar{\chi}}=0$, or $m_{\chi}=m_{\bar{\chi}}=1$. Hence, obviously, the conditions in Theorem \ref{Paola} are not satisfied.
Moreover, when $r=\dim\Z\geq 2$, the condition on the eigenspaces in Theorem \ref{Moha} is never satisfied. Hence the case of totally decomposable abelian $G$-covers does not fall within the already studied cases. Hence Theorem \ref{ThmA} goes in the direction of completing the proof of the necessity of \eqref{star} in the abelian case.

\end{document}